\documentclass[12pt]{article}
\usepackage[a4paper,top=28mm,bottom=28mm,inner=30mm,outer=30mm]{geometry}
 
\usepackage[svgnames,x11names]{xcolor}
\usepackage{amsfonts,amsmath,amsthm,boldline,colortbl,comment,enumerate,float,graphicx,import,makecell,mathtools,multirow,pdflscape,subcaption,tikz}
\mathtoolsset{centercolon}
\usetikzlibrary{through,intersections,calc}
\usepackage[hidelinks]{hyperref}
\usepackage{physics}
\usepackage{cite}

\newcommand{\wt}{\textnormal{wt}}

\newcommand{\rk}{\textnormal{rk}}

\newcommand{\spac}{\mathbb{F}^{\mathbf{n} \times \mathbf{m}}_q}
\newcommand{\spacc}[3]{\mathbb{F}^{\mathbf{#1}\times \mathbf{#2}}_{#3}}
\newcommand{\srk}{\mathrm{srk}}

\newcommand{\erase}[2][says]

\newlength{\mynodespace}
\newtheorem{theorem}{Theorem}
\newtheorem{lemma}[theorem]{Lemma}
\newtheorem{corollary}[theorem]{Corollary}
\newtheorem{proposition}[theorem]{Proposition}

\theoremstyle{definition}

\newtheorem{definition}[theorem]{Definition}

\title{{Improved Gilbert-Varshamov bound\\ for sum-rank-metric codes via graph theory}}

\author{Aida Abiad\thanks{\texttt{a.abiad.monge@tue.nl}, Department of Mathematics and Computer Science, Eindhoven University of Technology, The Netherlands}
\thanks{Department of Mathematics and Data Science of Vrije Universiteit Brussel, Belgium}
\and
Harper Reijnders\thanks{\texttt{l.e.r.m.reijnders@tue.nl}, Department of Mathematics and Computer Science, Eindhoven University of Technology, The Netherlands}
\and
Michael Tait\thanks{\texttt{michael.tait@villanova.edu}, Department of Mathematics and Statistics, Villanova University, U.S.A.}
}
\date{}

\begin{document}

\maketitle

\begin{abstract}
We use a graph-theoretic approach which yields improvements on the known Gilbert-Varshamov (GV) bound for sum-rank-metric codes for certain parameters. In particular, we show that asymptotically $\mathbb{F}_q^{\mathbf{n}\times \mathbf{m}}$ can be partitioned into sum-rank-metric codes whose average size is bigger than the GV bound by a logarithmic factor for these parameters. Finally, we discuss the connection of such codes to set-coloring Ramsey numbers.

\bigskip


\end{abstract}

\bigskip

\section{Introduction}

In recent years, the theory of sum-rank-metric codes has been extensively developed \cite{byrne_fundamental_2021, byrne2022anticodes, martinez_penas_codes_2022,  puchinger_generic_2022, M2024, AAA2024, AGKP2025}, and sum-rank-metric codes have found applications in several areas such as distributed storage \cite{martinez-penas_universal_2019}, space-time codes \cite{shehadeh_spacetime_2022} and multi-shot networks \cite{martinez-penas_reliable_2019}.

A \emph{sum-rank-metric code} is a subset of the Cartesian product of matrix spaces over a finite field. The \emph{sum-rank weight} of an element of this space is defined as the sum of the ranks of the component matrices. The \emph{sum-rank distance} between two elements is subsequently defined as the sum of the ranks of the differences of the component matrices. In general, codes that have a large minimum distance are desirable, as they allow many errors to be corrected. Sum-rank-metric codes generalize both codes in the Hamming metric (when all matrices are $1\times 1$) and in the rank metric (when there is 1 component matrix \cite{gabidulin, rankcodebook}).  

A fundamental problem in coding theory, and in particular also in the study of sum-rank-metric codes, is that of finding the largest code for which the elements elements of the code are pairwise at minimum distance $d$, for some fixed $d$. This quantity is denoted by $A^{SRK}_{q}(\mathbf{n},\mathbf{m},d)$. Many different upper bounds on $A^{SRK}_{q}(\mathbf{n},\mathbf{m},d)$ are known, see \cite[Section III]{byrne_fundamental_2021} for an overview. More recently, stronger upper bounds obtained from spectral graph theory were shown in \cite{AAA2024}, and also a linear programming upper bound was obtained in \cite{AGKP2025}; this latter is the best upper bound currently known for sum-rank metric codes. On the other side, the most common way to find lower bounds on any coding metric is through constructions. For the sum-rank case too, many different types of sum-rank-metric codes have been constructed to obtain lower bounds \cite{AGcodes, chen2023, LCC, M2024, martinez2023}. In particular, much effort has been spent finding \emph{maximum sum-rank distance codes} (\emph{MSRD codes} for short), which are sum-rank-metric codes meeting the Singleton bound (see \cite[Theorem III.1]{byrne_fundamental_2021}). We refer the reader to \cite{M2024} for the details of various MSRD constructions.

In this paper we use graph-theoretic techniques to study sum-rank-metric codes. Our first result is an improvement of the well-known Gilbert-Varshamov (GV) lower bound, also known as the sphere covering bound. Such a GV bound can be obtained for any metric by a simple sphere covering argument. In the sum-rank-metric case, the GV bound was shown in \cite[Theorem III.11]{byrne_fundamental_2021}. For the case of linear sum-rank-metric codes, simplified and easier to compute versions of the GV bound were derived in \cite{OPB2021}, where it was also shown that random linear codes attain the GV bound asymptotically. Though they only considered the case where the sizes of the component matrices are equal. In \cite{chen2023}, explicit constructions of codes achieving the asymptotic GV bound from \cite{OPB2021} were given. 

In the first part of this paper we will derive asymptotic improvements of the general GV bound for sum-rank-metric codes (see Theorem \ref{thm:GV_SRK}) for certain parameters. In fact, we will show something stronger, namely that the sum-rank-metric ambient space may be partitioned into codes whose average size is bigger than the GV bound. We do so by building on work of Jiang and Vardy \cite{jiang_asymptotic_2004} and Vu and Wu \cite{VW}, who developed an approach for improving the GV bound for Hamming-metric codes by describing the ambient space in terms of a graph and bounding the local density. This graph-theoretic approach has been applied extensively, see \cite{jiang_asymptotic_2004, VW, zhang_improved_2023, kim_exponential_2022}.

Using this graph-theoretic approach, we extend the results of Vu and Wu to certain sum-rank metrics. In particular, we obtain an asymptotic improvement for when the number of components of the sum-rank metric tends to infinity (Theorem \ref{thm:srk_asymp_gv}), and for when the size of one of the matrices grows asymptotically (Theorem \ref{thm:matrix_size_gv}).

The second part of the paper uses a problem in Ramsey theory to study sum-rank-metric codes. A \emph{set-coloring Ramsey number}, denoted $R(k;r,s)$, is the minimum positive integer $n$ such that in every edge-coloring of the complete graph on $n$ vertices, where each edge is mapped onto a set of $s$ colors chosen from a palette of $r$ colors, there exists a monochromatic $k$-clique. A connection between set-coloring Ramsey numbers and Hamming-metric codes was shown in \cite{CFHMSV22}, and further expanded upon in \cite{CFPZ2024}. Here we establish a link between sum-rank-metric codes and set coloring Ramsey numbers using the aforementioned connection, as well as the relation between Hamming-metric codes and sum-rank-metric codes, see e.g.\cite{byrne_fundamental_2021, martinez_penas_codes_2022}.

This paper is structured as follows. In Section \ref{sec:preliminaries} the basic definitions for graphs and sum-rank-metric codes are presented. Then, in Section \ref{sec:ham_srk_connection}, we cover the connection between the sum-rank metric and the Hamming metric which will be used in the rest of the paper. In Section \ref{sec:asymp_GV} we provide an asymptotic improvement on the GV bound (Theorem \ref{thm:GV_SRK}) for certain sum-rank-metric codes. Finally, in Section \ref{sec:SRK_Ramsey} we explore a connection between sum-rank-metric codes and set-coloring Ramsey numbers.

\section{Preliminaries}
\label{sec:preliminaries}


In this paper we consider simple undirected graphs, which we denote as $G = (V,E)$, where $V$ is the vertex set, and $E \subset \binom{V}{2}$ the edge set. The \emph{degree} of a vertex $v \in V$ is defined as the number of edges incident to $v$. The \emph{maximum degree} of a graph, denoted by $D(G)$ or simply $D$, is defined as the maximum degree of a vertex over all vertices in $G$. The \emph{independence number} of $G$, denoted by $\alpha(G)$, is defined as the largest subset of vertices $U \subset V$, such that the vertices in $U$ are pairwise not adjacent. The {\em chromatic number} of $G$, denoted by $\chi(G)$, is the minimum number of sets required to partition the vertex set of $G$ into independent sets.

For a graph $G$ and a positive integer $k$, the \emph{$k$-th power graph} of $G$, denoted by $G^k$, has the same vertex set as $G$, with two vertices being adjacent in $G^k$ if they are at distance at most $k$ in the original graph $G$.

For a positive integer $n$, we denote by $[n]$ the set $\{1,\dots,n\}$. Let $t$ be a positive integer and let $\mathbf{n} = (n_1,\hdots,n_t)$, $\mathbf{m} = (m_1,\hdots,m_t)$ be ordered tuples of positive integers with $m_i\geq n_i$ for all $i\in [t]$. 

For a prime power $q$ and positive integers $m \geq n$, 
let $\mathbb{F}_q^{n\times m}$ denote the vector space of all $n \times m$ matrices over the finite field $\mathbb{F}_q$. 
Denote by $\rk(M)$ the rank of a matrix $M \in \mathbb{F}_q^{n \times m}$.

\begin{definition}\label{def:srk}
    The \emph{sum-rank-metric space} is 
    an $\mathbb{F}_q$-linear vector space 
    $\mathbb{F}_q^{\mathbf{n}\times \mathbf{m}}$ defined as the Cartesian product of the vector spaces $\mathbb{F}_q^{n_i \times m_i}$ as follows:
    \begin{equation*}
        \mathbb{F}_q^{\mathbf{n}\times \mathbf{m}}:=
    \mathbb{F}_q^{n_1\times m_1}\times \cdots \times
    \mathbb{F}_q^{n_t\times m_t}.
    \end{equation*}
    
    For an element $\mathbf{X} =(X_1,\ldots, X_t) \in \mathbb{F}_q^{\mathbf{n} \times \mathbf{m}}$, we call $\srk(\mathbf{X}) =  \sum_{i=1}^t \rk(X_i)$ the \emph{sum-rank} of $\mathbf{X}$.
    The \emph{sum-rank distance} between 
    $\mathbf{X},\mathbf{Y} \in \mathbb{F}_q^{\mathbf{n}\times \mathbf{m}}$ is $\srk(\mathbf{X} - \mathbf{Y})$, where we use the notation $\mathbf{X} - \mathbf{Y} = (X_1-Y_1,\hdots,X_t-Y_t)$.  
\end{definition}

\begin{definition}
A \emph{sum-rank-metric code} is a non-empty subset $\mathcal{C} \subseteq \mathbb{F}_q^{\mathbf{n}\times \mathbf{m}}$. 
The \emph{minimum distance} of a code $\mathcal{C}$ with $\abs{\mathcal{C}} \geq 2$ is defined as the minimum distance between two codewords $\mathbf{X},\mathbf{Y} \in \mathcal{C}$, and we denote it as:
\begin{equation*}
    \srk(\mathcal{C}) := \min\left\{\srk(\mathbf{X}-\mathbf{Y})\colon \mathbf{X}, \mathbf{Y} \in \mathcal{C}, \, \mathbf{X} \neq \mathbf{Y}\right\}.
\end{equation*}
\end{definition}

Note that if $t=1$, then the sum-rank metric reduces to the rank-metric, and if $n_i=m_i=1$ for all $i \in [t]$, then it reduces to the Hamming metric.

By $A^{SRK}_{q}(\mathbf{n},\mathbf{m},d)$ we denote the size of the largest code $\mathcal{C} \subset \spac$ such that $\srk(\mathcal{C}) \geq d$. I.e.
\begin{equation*}
    A^{SRK}_{q}(\mathbf{n},\mathbf{m},d) := \max\{\abs{\mathcal{C}} : \mathcal{C} \subset \spac, \srk(\mathcal{C}) \geq d\}.
\end{equation*}

We denote by $V_q(\mathbf{n},\mathbf{m},k)$ the \emph{volume of the ball of radius $k$} around a codeword in $\spac$
\begin{equation*}
    V_q(\mathbf{n},\mathbf{m},k) = \abs{\{\mathbf{X} \in \spac: \srk(\mathbf{X}) \leq k\}}.
\end{equation*}

The following definition, first introduced in \cite{AAA2024}, allows us to study sum-rank-metric codes from a graph-theoretical perspective. 

\begin{definition}\label{def:srkgraph}
    The \emph{sum-rank-metric graph} $\Gamma(\spac)$ is a graph whose vertex set is $\mathbb{F}_q^{\mathbf{n}\times \mathbf{m}}$, where two vertices $\mathbf{X},\mathbf{Y}\in \mathbb{F}_q^{\mathbf{n} \times \mathbf{m}}$ are adjacent if and only if $\srk(\mathbf{X}-\mathbf{Y}) = 1$. Such a graph has $\abs{V} = q^{\sum_{i=1}^t n_im_i}$ vertices.
\end{definition}

A simple yet important observation is that, for any two vertices $\mathbf{X},\mathbf{Y}$ of the graph $\Gamma\left(\mathbb{F}_q^{\mathbf{n}\times \mathbf{m}}\right)$, the geodesic distance between them coincides with the sum-rank distance $\srk(\mathbf{X}-\mathbf{Y})$, see \cite[Proposition 4.3]{byrne2022anticodes}. This observation gives us a useful characterization of adjacency in the power graph $\Gamma(\spac)^k$. Two vertices $\mathbf{X},\mathbf{Y}$ are adjacent in $\Gamma(\spac)^k$ if and only if $\srk(\mathbf{X}-\mathbf{Y}) \leq k$. Hence, we have the following result, which lies at the heart of our graph theoretical approach - an equivalence that allows us to use the graph theory toolkit to tackle a coding theoretical problem.

\begin{proposition}\cite[Corollary 16]{AAA2024}\label{prop:Aq_alpha}
    For a sum-rank-metric space $\spac$, and an integer $k$, we have
    \begin{equation*}
        A_q^{SRK}(\mathbf{n},\mathbf{m},k+1) = \alpha(\Gamma(\spac)^k).
    \end{equation*}
\end{proposition}

\section{The connection between sum-rank-metric codes and Hamming codes}\label{sec:ham_srk_connection}

    As has already been discussed, sum-rank-metric codes are a generalization of Hamming metric codes. We will exploit this (and a more general) connection in the upcoming Sections \ref{sec:asymp_GV} and \ref{sec:SRK_Ramsey}. In this section, we will state and explain the results used in these later sections.
    
    The \emph{Hamming weight} of a vector $\mathbf{X} \in \mathbb{F}_q^N$, denote by $\wt(\mathbf{X})$, is defined as the number of nonzero elements in $\mathbf{X}$. This gives rise to a distance, which is called the \emph{Hamming distance}, for $\mathbf{X},\mathbf{Y} \in \mathbb{F}_q^N$ we define $d_H(\mathbf{X},\mathbf{Y}) = \wt(\mathbf{X}-\mathbf{Y})$. We further define, for a Hamming-metric code $C \subset \mathbb{F}_q^N$, $d_H(C) = \min\{d_H(\mathbf{X},\mathbf{Y}): \mathbf{X},\mathbf{Y} \in C\}$, and $A_q(N,d) = \max\{\abs{C}: C \subset \mathbb{F}_q^N, d_H(C) \geq d\}$.
    
    When in the sum-rank metric $n_i=m_i=1$ for all $i \in [t]$, then we indeed recover the Hamming metric. On top of this case, there are more interesting scenarios where the sum-rank metric can be connected to the Hamming metric. This was done in for instance \cite{byrne_fundamental_2021} and \cite{martinez_penas_codes_2022}.

    For $m = \max\{m_1,\hdots,m_t\}$, consider a basis $\mathbf{\alpha} = (\alpha_1,\hdots,\alpha_m)$ of $\mathbb{F}_{q^m}$ as a vector space over $\mathbb{F}_q$. Then let us define the map $f:\mathbb{F}_q^{\mathbf{n} \times \mathbf{m}} \rightarrow \mathbb{F}_{q^m}^N$ as:
    \begin{equation}\label{eq:srk_to_Hamming}
        f(\mathbf{X}) = (f^1(X_1),\hdots,f^t(X_t)),
    \end{equation}
    where $f^i: \mathbb{F}_q^{n_i \times m_i} \rightarrow \mathbb{F}_{q^m}^{n_i}$ is defined as:
    \begin{equation*}
        f^i(X_i) = \left(\sum_{j=1}^m (X_i)_{1,j}\alpha_j, \hdots, \sum_{j=1}^m (X_i)_{t,j}\alpha_j\right).
    \end{equation*}
    Notably,
    \begin{equation}\label{eq:codeword_leq}
        \srk(\mathbf{X}) \leq \wt_H(f(\mathbf{X})),
    \end{equation}
    as observed in \cite{byrne_fundamental_2021}. Inequality \eqref{eq:codeword_leq} was used to prove the following result:
    \begin{theorem}\cite[Theorem III.1]{byrne_fundamental_2021}\label{thm:srk_ham_leq}
        \begin{equation}\label{eq:leq_ineq}
            A_q^{SRK}(\mathbf{n},\mathbf{m},d) \leq A_{q^m}(N,d). 
        \end{equation}
    \end{theorem}

    The next result gives a direct connection between the sum-rank metric and the Hamming metric in certain cases.

    \begin{theorem}\cite[Theorem 1.2]{martinez_penas_codes_2022}\label{thm:srk_ham_rr}
        Let $\mathbf{m} = (m,\hdots,m)$ and $f$ be as in \eqref{eq:srk_to_Hamming}, for some basis $\mathbf{\alpha} \subset \mathbb{F}$. Consider and  $\mathbf{X} \in \spac$, and for some $\mathbf{A} = (A_1,\hdots,A_t)$ with $A_i \in \mathbb{F}^{n_i \times n_i}$, denote by $\mathbf{X}\mathbf{A} = (X_1A_1,\hdots,X_tA_t)$. Then
        \begin{equation*}
            \srk(\mathbf{X}) = \min\{\wt_H(f(\mathbf{X}\mathbf{A}) \mid A_i\in \mathbb{F}^{n_i \times n_i} \text{ invertible for } i \in [t]\}.
        \end{equation*}        
    \end{theorem}
    Next we show that in fact a slightly more general equality with the Hamming metric follows from Theorem \ref{thm:srk_ham_rr}.
    
\begin{proposition}\label{prop:all1_allm}
    Let $\mathbf{n} = (1,\hdots,1)$ and $\mathbf{m} = (m,\hdots,m)$ for some $m \in \mathbb{N}$. Then $\spac \simeq \mathbb{F}_{q^m}^t$, and hence
    \begin{equation*}
         A_{q}^{SRK}(\mathbf{n},\mathbf{m},d) = A_{q^m}(N,d).
    \end{equation*}
\end{proposition}
\begin{proof}
    Consider the map $f$ as in \eqref{eq:srk_to_Hamming}. Note that since all entries of $\mathbf{m}$ are the same, $f$ is bijective (as observed in \cite[Theorem III.1]{byrne_fundamental_2021}). Thus, what remains to show is that $f$ also preserves the weight. That is, for any $\mathbf{X} \in \spac$, we must have that $\srk(\mathbf{X}) = \wt_H(f(\mathbf{X}))$. This we can deduce from Theorem \ref{thm:srk_ham_rr}, as the only invertible $1 \times 1$ matrices are simply scalar multiples of the identity $I$. Hence, if we denote $\mathbf{I} = (I_{n_1}, \hdots, I_{n_t})$, where $I_n$ is the $n \times n$ identity matrix, we have
    \begin{align*}
        \srk(\mathbf{X}) &= \min\{\wt_H(f(\mathbf{X}\mathbf{A}) \mid A_i\in \mathbb{F}^{n_i \times n_i} \text{ invertible for } i \in [t]\}  \\
        &=  \wt_H(f(\alpha \mathbf{X}\mathbf{I}) = \wt_H(f(\mathbf{X})). \qedhere
    \end{align*}
\end{proof}

As a result of Proposition \ref{prop:all1_allm}, we can give a Hamming metric space for which the opposite type of inequality to \eqref{eq:leq_ineq} holds.

\begin{corollary}\label{cor:srk_ham_geq}
    Let $\mathbf{n} = (n_1,\hdots,n_t)$ and $\mathbf{m} = (m_1,\hdots,m_t)$ with $m' = \min\{m_i\}$. Then
    \begin{equation*}
         A_{q}^{SRK}(\mathbf{n},\mathbf{m},d) \geq A_{q^{m'}}(t,d).
    \end{equation*}
\end{corollary}
\begin{proof}
    Consider a code $C \subset \mathbb{F}_{q^{m'}}^t$. Let $\mathbf{n'} = (1,\hdots,1)$ and $\mathbf{m'} = (m',\hdots,m')$ then by Proposition \ref{prop:all1_allm}, we find a code $C' \in \spacc{n'}{m'}{q}$ with $d_H(C) = \srk(C')$ and $\abs{C} = \abs{C'}$. Then since $\spacc{n'}{m'}{q}$ is contained in $\spac$, $C'$ induces a code $C''$ in $\spac$ with $\srk(C') = \srk(C'')$ and $\abs{C'} = \abs{C''}$.
\end{proof}

We will use these various ways the Hamming metric and sum-rank metric are connected in the following Sections \ref{sec:asymp_GV} and \ref{sec:SRK_Ramsey}. 

\section{Asymptotic improvements on the GV bound}\label{sec:asymp_GV}

As already mentioned, the GV bound is a classical bound in coding theory, which can be derived for any coding metric. This was first done for sum-rank-metric codes in \cite{byrne_fundamental_2021}.

\begin{theorem}{[Gilbert-Varshamov bound, \cite[Theorem III.11]{byrne_fundamental_2021}]}\label{thm:GV_SRK}
    For a sum-rank-metric space $\spac$, let $V_q(\mathbf{n},\mathbf{m},k)$ be the volume of the ball of radius $k$ around a codeword. Then
    \begin{equation*}
        A_q^{SRK}(\mathbf{n},\mathbf{m},k+1) \geq \frac{\abs{V}}{V_q(\mathbf{n},\mathbf{m},k)}.
    \end{equation*}
\end{theorem}

In this section, we will give an asymptotic improvement of Theorem \ref{thm:GV_SRK} for some sum-rank metrics through a graph theoretical approach. In fact, we will show the stronger result that $\spac$ can be partitioned into sum-rank-metric codes with average size larger than the above GV bound.

Much like for coding metrics, there exists a sphere covering bound for graphs as well. Indeed, if we consider the volume of the ball of radius $1$ around a vertex, then this will be bounded by $D+1$. Hence, by a sphere covering argument we obtain
\begin{equation}\label{eq:graph_sphere_covering}
    \alpha(G) \geq \frac{\abs{V}}{D+1},
\end{equation}
where $D$ denotes the maximum degree of $G$. In \cite{AKS} the bound from \eqref{eq:graph_sphere_covering} was improved by a factor of $\Omega(\log D)$ for triangle-free graphs. This was later extended to graphs that contain relatively few triangles as follows.

\begin{theorem}\cite[Lemma 15]{bollobasbook}, \label{thm:alpha_logbound}
    Let $G = (V,E)$ be a graph with maximum degree $D$, and suppose that $G$ contains no more than $\Delta$ triangles. Then, 
\begin{equation*}
    \alpha(G) \geq \frac{\abs{V}}{10D}\left( \log_2 D - \frac{1}{2}\log_2\left(\frac{\Delta}{\abs{V}}\right)\right).
\end{equation*}
\end{theorem}

The bound from Theorem \ref{thm:alpha_logbound} is the main tool we will use for our GV bound improvement. In particular, if for some positive $\varepsilon \leq 2$ we have $\frac{\Delta}{\abs{V}} \leq D^{2-\varepsilon}$ for the graph $\Gamma(\spac)^k$, then the bound from Theorem \ref{thm:alpha_logbound} improves the classic Gilbert-Varshamov bound from Theorem \ref{thm:GV_SRK} by a multiplicative factor of $\log_2D$.

Moreover, in order to obtain the stronger result on partitioning $\spac$ into sum-rank-metric codes, we will also make use of the following upper bound on the chromatic number of a graph, which extends Theorem \ref{thm:alpha_logbound} to a bound on the chromatic number.

\begin{theorem}\cite[Theorem 1]{AKSchromatic}\label{thm:chi_logbound}
    There exist an absolute positive constant $c$ such that the following holds. For any graph $G = (V,E)$ with maximum degree $D$ and at most $D^2/f$ edges in the neighborhood of any vertex $v$ for 
    some $f$, it holds that
    \begin{equation*}
        \chi(G) \leq c \frac{D}{\log f}.
    \end{equation*}
\end{theorem}

We note that Theorem \ref{thm:chi_logbound} also implies a logarithmic strengthening of the GV bound. This follows directly when $f = D^\varepsilon$, using the classical inequality $\alpha(G) \geq \frac{\abs{V}}{\chi(G)}$.

In order to show that $\Gamma(\spac)^k$ satisfies the condition to apply Theorems \ref{thm:alpha_logbound} and \ref{thm:chi_logbound}, we first need a few things to make it easier to work with. In particular, we want $\Gamma(\spac)^k$ to be vertex-transitive, as in this case $3\frac{\Delta}{\abs{V}} = T$, where $T$ denotes the number of edges in a neighborhood of an arbitrary vertex. We will show this property by proving that $\Gamma(\spac)^k$ is a Cayley graph, similarly as how it was done for $k=1$ in \cite[Proposition 10]{AAA2024}. Recall that a graph is a \emph{Cayley graph} over a group $G$ with connecting set $S \subset G$ if the vertices of the graph correspond to the elements of $G$, such that two vertices $x,y$ are adjacent in the graph if and only if there is an element $s \in S$ for which $x+s = y$. Note that if $S$ is symmetric then the graph is undirected, and the Cayley graphs we consider will be generated by symmetric sets.
\begin{proposition}\label{prop:Cayley}
    The graph $\Gamma(\spac)^k$ is a Cayley graph over the additive group $\spac$ with connecting set $S = \{\mathbf{X} \in \spac \mid \srk(\mathbf{X}) \in \{1,\hdots,k\}\}$. In particular, $\Gamma(\spac)^k$ is vertex-transitive.
\end{proposition}
\begin{proof}
    The proof is analogous to \cite[Proposition 10]{AAA2024}. Consider two adjacent vertices $\mathbf{X}$ and $\mathbf{Y}$ in $\Gamma(\spac)^k$, these will be at distance at most $k$ in the original graph $\Gamma(\spac)$. Recall, by \cite[Proposition 4.3]{byrne2022anticodes}, the geodesic distance in $\Gamma(\spac)$ is equal to the sum-rank-metric weight of $\mathbf{X}-\mathbf{Y}$, hence $\srk(\mathbf{X}-\mathbf{Y}) \in \{1,\hdots,k\}$. Thus $\mathbf{X}-\mathbf{Y} \in S$. Now take a $\mathbf{Z} \in S$ and consider the vertices $\mathbf{X}$ and $\mathbf{X}+\mathbf{Z}$. Clearly $\srk(\mathbf{X}+\mathbf{Z}-\mathbf{X}) = \srk(\mathbf{Z}) \in \{1,\hdots,k\}$. Hence, we know the distance between $\mathbf{X}+\mathbf{Z}$ and $\mathbf{X}$ in $\Gamma(\spac)$ is at most $k$, and subsequently $\mathbf{X}+\mathbf{Z}$ and $\mathbf{X}$ must be adjacent in $\Gamma(\spac)^k$.
\end{proof}

Since the power of the sum-rank-metric graph $\Gamma(\spac)^k$ is vertex-transitive, it is also regular. In the sequel, denote by $D$ the degree of regularity of $\Gamma(\spac)^k$. Note that in general, the expression for $D$ is quite involved, see \cite[Lemma III.5]{byrne_fundamental_2021}. Since our graph is vertex-transitive, we have $3\frac{\Delta}{\abs{V}} = T$, and thus we only need to find an $\varepsilon > 0 $ such that $T \leq D^{2-\varepsilon}$, as we can then apply Theorem \ref{thm:alpha_logbound} and Theorem \ref{thm:chi_logbound}. In the rest of this section, we present two approaches to obtain this inequality between $T$ and $D$ for two different kind of asymptotic versions of the sum-rank metric. First, we will consider sum-rank metrics which grow in the number of components, then we will consider sum-rank metrics where one of the component matrices grows in size.

\subsection{Asymptotic improvement in the number of components}
In this section, we will consider an asymptotic improvement for the GV bound when $t$ (the number of components of the sum-rank metric) goes to infinity.

Rather than showing $T \leq D^{2-\varepsilon}$ directly, we will use the fact that this has already been proven for the Hamming case in \cite{VW}. In section \ref{sec:ham_srk_connection}, we saw that the sum-rank metric and Hamming metric can be related to each other. In particular, we will make use of Proposition \ref{prop:all1_allm}.

Define the \emph{Hamming graph} as $H(t,q^m) = \Gamma(\spacc{\mathbf{1}}{\mathbf{m}}{q})$, with $\mathbf{1} = (1,\hdots,1)$ and $\mathbf{m} = (m,\hdots,m)$, both of which are of length $t$.

\begin{lemma}\cite{VW}\label{lem:DH_TH_ineq}
    Let $D_H$ denote the regularity of $H(t,q^m)^k$, and let $T_H$ be the number of edges in a neighborhood of $H(t,q^m)^k$. Let $\alpha,\alpha'$ be two constants satisfying $0 < \alpha' < \alpha < \frac{q^m-1}{q^m}$. Then for any $\alpha't \leq k \leq \alpha t$ with $t$ sufficiently large, there exists a positive $\varepsilon \leq 2$ for which $T_H \leq D_H^{2-\varepsilon}$.
\end{lemma}

To apply Lemma \ref{lem:DH_TH_ineq} to our case, we make use of a similar approach as was used to prove this lemma originally. Namely, we use the notion of polynomial equivalence.

Let $f$ and $g$ be two functions in $n$. We say $f$ and $g$ are \emph{polynomially equivalent}, denoted by $f \sim g$, if there are positive constants $c_1,c_2$ such that
\begin{equation*}
    n^{-c_1}f \leq g \leq n^{c_2}f.
\end{equation*}

In particular, if we can show $D \sim D_H$ and $T \sim T_H$, then by Lemma \ref{lem:DH_TH_ineq} (and the fact that $D$ and $T$ are exponential in $t$), we can use Theorem \ref{thm:alpha_logbound} and Theorem \ref{thm:chi_logbound} to obtain the following result:

\begin{theorem}\label{thm:srk_asymp_gv}
    Let $\mathbf{n} = (n_1,\hdots),\mathbf{m} = (m_1,\hdots)$ be two infinite sequences of $\mathbb{N}$ valued functions of $t$, and let $t^*$ be another $\mathbb{N}$ valued function of $t$ such that:
    \begin{itemize}
        \item $m_i \geq n_i$ for all $i,t \in \mathbb{N}$.
        \item $n_i = 1$ and $m_i = m'$ if $i > t^*$ for all $t \in \mathbb{N}$, where $m'$ is a function in $t$.
        \item $n = \max\{n_i : i \in \mathbb{N}\}$ and $m = \max\{m_i : i \in \mathbb{N}\}$ are well defined functions of $t$.
        \item $mnt^* = O(\log_q t)$.
    \end{itemize}
    Define $\mathbf{n_t} = (n_1,\hdots,n_t),\mathbf{m_t} = (m_1,\hdots,m_t)$. Let $\alpha,\alpha'$ be two constants satisfying $0 < \alpha' < \alpha < \frac{q^{m'}-1}{q^{m'}}$. Then, for sufficiently large $t$, there is a positive constant $\varepsilon \leq 2$ depending on $q^{m'}$ and $\alpha$ such that for any $\alpha't \leq k \leq \alpha t$ we have
    \begin{equation*}
        A_q^{SRK}(\mathbf{n_t},\mathbf{m_t},k+1) \geq \varepsilon\frac{\abs{V}}{20D}\log_2 D.
    \end{equation*}
    In fact, the whole space $\spacc{n_t}{m_t}{q}$ can be partitioned into sum-rank-metric codes of minimum distance $k$, such that the average size of these codes is the above bound.
\end{theorem}
\begin{proof}
    Let $D$ and $T$ denote the regularity and number of edges in a neighborhood of a vertex of $\Gamma(\spacc{n_t}{m_t}{q})^k$, respectively. Further, let $D_H$ denote the regularity of $H(t-t^*,q^{m'})^k$, and let $T_H$ be the number of edges in a neighborhood of $H(t-t^*,q^{m'})^k$.

    We know from Lemma \ref{lem:DH_TH_ineq} that there is a positive constant $\varepsilon$ for which $T_H \leq D_H^{2-\varepsilon}$ if $t$ is sufficiently large. Our strategy will thus be to show that $D \sim D_H$ and $T \sim T_H$, after which the result follows from Theorems \ref{thm:alpha_logbound} and \ref{thm:chi_logbound}.

    First we look at $D$. By regularity, we can find $D$ by counting the neighbors of the all $0$ element. For this we need to count the number of $\mathbf{X} \in \spacc{n_t}{m_t}{q}$ for which $\srk(\mathbf{X}) \leq k$. If we define $\mathbf{X}' = (X_{t^*+1},\hdots,X_t) \in \spacc{n_t'}{m_t'}{q}$, where $\mathbf{n_t'} = (1,\hdots,1), \mathbf{m_t'} = (m_{t^*+1},\hdots,m_t) \in \mathbb{N}^{t-t^*}$ then for given $X_1,\hdots,X_{t^*}$, $\mathbf{X}'$ needs to satisfy $\srk(\mathbf{X}') \leq k -\sum_{i=1}^{t^*} \rk(X_i)$. We will thus focus on bounding the amount of such $\mathbf{X}'$, and then multiplying this by the amount of choices for $X_1,\hdots,X_{t^*}$. By Proposition \ref{prop:all1_allm} the amount of $\mathbf{X}'$ with $\srk(\mathbf{X}') \leq k$, is equal to the amount of $\tilde{\mathbf{X}} \in \mathbb{F}_{q^{m'}}^{t-t^*}$ with $\wt_H(\tilde{\mathbf{X}}) \leq k$. Hence:
    \begin{align*}
        \abs{\left\{\mathbf{X}' \in \spacc{n_t'}{m_t'}{q} : \srk(\mathbf{X}') \leq k \right\}} = \abs{\left\{\tilde{\mathbf{X}} \in \mathbb{F}_{q^{m'}}^{t-t^*} : \wt_H(\tilde{\mathbf{X}}) \leq k \right\}} = D_H.
    \end{align*}    
    The above holds regardless of choice of $X_1,\hdots,X_{t^*}$. On the other hand, $\mathbb{F}_{q^{m'}}^{t-t^*}$ is clearly a smaller space than $\spacc{n_t}{m_t}{q}$. Thus, we can bound
    \begin{align*}
        D_H &\leq D\\
        &\leq \abs{\{(X_1,\hdots,X_{t^*}) \in \mathbb{F}_q^{(n_1,\hdots,n_{t^*}) \times (m_1,\hdots,m_{t^*})}\}} D_H\\
        &= q^{\sum_{i=1}^{t^*} n_im_i}D_H.
    \end{align*}

    Since $n_i \leq n, m_i \leq m$ and we assumed $mnt^* = O(\log_q t)$, we have that $q^{\sum_{i=1}^{t^*} n_im_i}$ is polynomial in $t$, and hence $D \sim D_H$.

    Next we look at $T$, which is exactly the number of pairs $(\mathbf{X},\mathbf{Y}) \in \left(\spacc{n_t}{m_t}{q}\right)^2$ for which:
    \begin{enumerate}
        \item $\srk(\mathbf{X}) \leq k$.
        \item $\srk(\mathbf{Y}) \leq k$.
        \item $\srk(\mathbf{X}-\mathbf{Y}) \leq k$.
    \end{enumerate}
    To show that  $T \sim T_H$, we will use an analogous strategy as we did for showing $D \sim D_H$. If we take $\mathbf{X}',\mathbf{Y}' \in \spacc{n_t'}{m_t'}{q}$, where $\mathbf{n_t'} = (1,\hdots,1), \mathbf{m_t'} = (m_{t^*+1},\hdots,m_t) \in \mathbb{N}^{t-t^*}$, then these must satisfy
    \begin{enumerate}
        \item $\srk(\mathbf{X}') \leq k-\sum_{i=1}^{t^*} \rk(X_i)$.
        \item $\srk(\mathbf{Y}') \leq k-\sum_{i=1}^{t^*} \rk(Y_i)$.
        \item $\srk(\mathbf{X}'-\mathbf{Y}') \leq k - \sum_{i=1}^{t^*} \rk(X_i-Y_i)$.
    \end{enumerate}
    We can now repeat the same trick by considering $\tilde{\mathbf{X}},\tilde{\mathbf{Y}} \in \mathbb{F}_{q^{m'}}^{t-t^*}$ to find that the amount of such pairs ($\tilde{\mathbf{X}},\tilde{\mathbf{Y}}$) is bounded by $T_H$, regardless of what we choose for the $X_1,\hdots,X_{t^*},Y_1,\hdots,Y_{t^*}$. Hence, we can bound $T$ as follows
    \begin{align*}
        T &\leq \abs{\{((X_1,\hdots,X_{t^*}),(Y_1,\hdots,Y_{t^*})) \in \left(\mathbb{F}_q^{(n_1,\hdots,n_{t^*}) \times (m_1,\hdots,m_{t^*})}\right)^2\}} T_H\\
        &= \abs{\{(X_1,\hdots,X_{t^*}) \in \mathbb{F}_q^{(n_1,\hdots,n_{t^*}) \times (m_1,\hdots,m_{t^*})}\}}^2 T_H \\
        &= q^{2 \sum_{i=1}^{t^*} n_im_i} T_H,
    \end{align*}
    and thus $T \sim T_H$, where again we used the assumption $mnt^* = O(\log_q t)$.

    By Lemma \ref{lem:DH_TH_ineq} there exists a positive $\varepsilon' > 0$ for which $T_H \leq D_H^{2-\varepsilon'}$, and subsequently we know there exists an $\varepsilon$ for which $T \leq D^{2-\varepsilon}$. Now, using Proposition \ref{prop:Aq_alpha} and Theorem \ref{thm:alpha_logbound}, we obtain
    \begin{align*}        
        A_q^{SRK}(\mathbf{n_t},\mathbf{m_t},k+1) &= \alpha\left(\Gamma\left(\spacc{n_t}{m_t}{q}\right)^k\right) \geq \frac{\abs{V}}{10D}\left( \log_2 D - \frac{1}{2}\log_2\left(\frac{\Delta}{\abs{V}}\right)\right) \\
        &= \frac{\abs{V}}{10D}\left( \log_2 \left(\frac{3D}{T^{\frac{1}{2}}}\right) \right) \geq \frac{\abs{V}}{10D}\left( \log_2 \left(\frac{D}{D^{\frac{1}{2}(2-\varepsilon)}}\right)\right)\\
        &=  \frac{\abs{V}}{10D}\left( \log_2 \left(\frac{1}{D^{-\frac{\varepsilon}{2}}}\right)\right)= \varepsilon\frac{\abs{V}}{20D}\log_2 D.      
    \end{align*}

    The stronger last statement about partitioning $\spacc{n_t}{m_t}{q}$ into sum-rank-metric codes whose average size is larger than this bound follows directly from Theorem \ref{thm:chi_logbound}, taking $f=D^\varepsilon$.
\end{proof}

\subsection{Asymptotic improvement in the size of the matrix}\label{sec:asymptoticsizematrix}

In this section we will consider the sum-rank metric with $\mathbf{n} = (n,n_2,\hdots,n_t)$ and $\mathbf{m} = (n,m_2,\hdots,m_t)$, where $n$ grows, and the $n_i,m_i$ and $t$ grow slow relative to $n$. The main result of this section is Theorem \ref{thm:matrix_size_gv}, which is an asymptotic strengthening of the GV bound from Theorem \ref{thm:GV_SRK}.

\begin{theorem}\label{thm:matrix_size_gv}
    Let $\mathbf{n} = (n,n_2,\hdots)$, $\mathbf{m} = (n,m_2,\hdots)$ be two infinite sequences of $\mathbb{N}$ valued functions of $n$, and let $t$ be another $\mathbb{N}$ valued function of $n$ such that:
    \begin{itemize}
        \item $m_\ell \geq n_\ell$ for all $\ell,n \in \mathbb{N}$.
        \item $\sum_{\ell=2}^t n_\ell m_\ell = O(\log_q n)$.
    \end{itemize}
    Define $\mathbf{n_n} = (n,n_2,\hdots,n_{t(n)}),\mathbf{m_n} = (n, m_2,\hdots,m_{t(n)})$. Further, suppose $k = \alpha n $ for some $0 < \alpha < \frac{2}{3}$. Then, for sufficiently large $n$, there is a positive constant $\varepsilon \leq 2$ depending on $q$ and $\alpha$ such that
    \begin{equation*}
        A_q^{SRK}(\mathbf{n_n},\mathbf{m_n},k+1) \geq \varepsilon\frac{\abs{V}}{20D}\log_2 D.
    \end{equation*}

\noindent    In fact, the whole space $\spacc{n_n}{m_n}{q}$ can be partitioned into sum-rank-metric codes of minimum distance $k$, such that the average size of these codes is the above bound.
\end{theorem}

\medskip

Our proof strategy for Theorem \ref{thm:matrix_size_gv} is once again to use Theorem \ref{thm:alpha_logbound} and Theorem \ref{thm:chi_logbound}. For this we will need $T \leq D^{2-\varepsilon}$ for some $\varepsilon>0$. We will be proving this in several steps, and we will make use of the following $3$ important quantities:
\begin{itemize}
    \item We define $M_q^n(k)$ to be the amount $n \times n$ matrices of rank $k$ over $\mathbb{F}_q$. See e.g. \cite[Section 1.7]{morrison_integer_2006} for the exact formula:
    \begin{equation*}
        M_q^n(k) = \prod_{\ell=0}^{k-1}\frac{ (q^{n}-q^\ell)^2}{(q^k-q^\ell)}.
    \end{equation*}
    \item We define $P_q^n(i,j,k)$ as the amount of pairs $X,Y \in \mathbb{F}_q^{n \times n}$ for which
    \begin{enumerate}
        \item $\rk(X) = i$.
        \item $\rk(Y) = j$.
        \item $\rk(X-Y) \leq k$.
    \end{enumerate}
    \item We define $Q_q^n(i,j,c)$ as the amount of matrices $Y \in \mathbb{F}_q^{n \times n}$ of rank $j$, such that for a fixed $X \in \mathbb{F}_q^{n \times n}$ of rank $i$, the dimension of $\mathrm{col}(X) \cap \mathrm{col}(Y)$ is exactly $c$. It follows that by summing $Q_q^n(i,j,c)$ over all $c$, we recover $M_q^n(j)$:
    \begin{equation*}
        M_q^n(j) = \sum_{c=0}^j Q_q^n(i,j,c).
    \end{equation*}
\end{itemize}

We now outline the general proof idea of Theorem \ref{thm:matrix_size_gv}
\begin{description}
    \item[Step 1] We bound $D(\Gamma(\spac)^k)$ from below by $M_q^n(k)$.
    \item[Step 2] We bound $T(\Gamma(\spac)^k)$ from above in terms of $P_q^n(i,j,k)$ (Lemma \ref{lem:TR_Rprime_ineq}).
    \item[Step 3] In Corollary \ref{cor:P_upper}, we obtain a bound on $P_q^n(i,j,k)$
    \begin{equation*}
        P_q^n(i,j,k) \leq M_q^n(i) \cdot 2\sum_{c=\lceil\frac{i+j-k}{2}\rceil}^jQ_q^n(i,j,c),
    \end{equation*}
    using the existing Lemma \ref{lem:col_row_rank}.
    \item[Step 4] We derive a closed expression for $Q_q^n(i,j,c)$ (Corollary \ref{cor:Q_expression}) using Lemma \ref{lem:intersect_spaces}.
    \item[Step 5] We show that for $i+j > k$ we have $Q_q^n(i,j,c) \leq (M_q^n(j))^\delta$ for a positive $\delta<1$ depending only on $\alpha$ and $q$ (Lemma \ref{lem:Q_eps}).
    \item[Step 6] We show that $P_q^n(i,j,k) < (M_q^n(k))^{2-\varepsilon}$, and subsequently that $T(\Gamma(\spac)^k)$ must be less than $M_q^n(k)^{2-\varepsilon}$ for some positive $\varepsilon$.
\end{description}

We start with step 1. The inequality $D(\Gamma(\spac)^k) \geq M_q^n(k)$ follows directly from the fact that for an $X \in \mathbb{F}_q^{n \times n}$, the matrix tuple $(X,O,\hdots,O) \in \spac$ is at distance $k$ from the all $0$ element. Thus each such $X$ defines a unique neighbor in $\spac$, and hence the degree must at least be $M_q^n(k)$.

Let us now proceed with the second step: bounding $T(\Gamma(\spac)^k)$ in terms of $P_q^n(i,j,k)$. 

\begin{lemma}\label{lem:TR_Rprime_ineq} 
    Let $\mathbf{n} = (n,n_2,\hdots,n_t)$, $\mathbf{m} = (n,m_2,\hdots,m_t)$. Then   \begin{equation*}\label{eq:TR_Rprime_ineq}
    T(\Gamma(\spac)^k) \leq 2q^{2\sum_{\ell=2}^t n_\ell m_\ell} \sum_{i=1}^k\sum_{j=1}^i P_q^n(i,j,k).
    \end{equation*}
\end{lemma}
\begin{proof}
First, note that $T(\Gamma(\spac)^k)$ is equal to the amount of pairs $\mathbf{X},\mathbf{Y} \in \spac$ for which
\begin{enumerate}
    \item $\srk(\mathbf{X}) \leq k$.
    \item $\srk(\mathbf{Y}) \leq k$.
    \item $\srk(\mathbf{X}-\mathbf{Y}) \leq k$.
\end{enumerate}

Denote $\mathbf{n}' = (n_2,\hdots,n_t)$, $\mathbf{m}' = (m_2,\hdots,n_t)$ and split $\mathbf{X}$ and $\mathbf{Y}$ into $\mathbf{X} = (X,\mathbf{X}')$ and $\mathbf{Y} = (Y,\mathbf{Y}')$, where $X,Y \in \mathbb{F}_q^{n \times n}$ and  $\mathbf{X}',\mathbf{Y}' \in \spacc{n'}{m'}{q}$. Clearly
\begin{equation*}
    T(\Gamma(\spac)^k) \leq T(\Gamma(\mathbb{F}_q^{n \times n})^k) \cdot T(\Gamma(\spacc{n'}{m'}{q})^k). 
\end{equation*}
Now consider $T(\Gamma(\spacc{n'}{m'}{q})^k)$, this is of course less than the total amount of pairs $(\mathbf{X}',\mathbf{Y}') \in \spacc{n'}{m'}{q} \times \spacc{n'}{m'}{q}$. Hence, we get the upper bound
\begin{equation*}
    T(\Gamma(\spacc{n'}{m'}{q})^k) \leq q^{2\sum_{\ell=2}^t n_\ell m_\ell}.
\end{equation*}

We now switch our focus to $T(\Gamma(\mathbb{F}_q^{n \times n})^k)$. We write:
\begin{equation*}
    T(\Gamma(\mathbb{F}_q^{n \times n})^k) = \sum_{i=1}^k\sum_{j=1}^k P_q^n(i,j,k).
\end{equation*}

Note that $P_q^n(i,j,k) = P_q^n(j,i,k)$, thus
\begin{equation*}
    T(\Gamma(\mathbb{F}_q^{n \times n})^k) \leq 2\sum_{i=1}^k\sum_{j=1}^i P_q^n(i,j,k). 
\end{equation*}
In conclusion,
\begin{align*}
    T(\Gamma(\spac)^k) &\leq T(\Gamma(\mathbb{F}_q^{n \times n})^k) \cdot T(\Gamma(\spacc{n'}{m'}{q})^k) \\
    &\leq 2 \sum_{i=1}^k\sum_{j=1}^i P_q^n(i,j,k)\cdot q^{2\sum_{\ell=2}^t n_\ell m_\ell}.\qedhere
\end{align*}
\end{proof}
There are two major benefits of using the bound from Lemma \ref{lem:TR_Rprime_ineq}. The first, is that the factor in front of the sum is polynomial in $n$, so it grows much slower than $M_q^n(k)$, and thus asymptotically in $n$ we do not have to worry about this multiplicative term. The second benefit, is that we can assume $i \geq j$ as we bound $P_q^n(i,j,k)$ in the sequel.

Next we will be bounding $P_q^n(i,j,k)$ (step 3). To do this, we make use of Lemma \ref{lem:col_row_rank}, which will allows us to bound the rank of the difference of two matrices in terms of the dimension of the intersection of their row and column spaces.

\begin{lemma}\cite[Section 4, Theorem 1]{marsaglia_bounds_1964}\label{lem:col_row_rank}
    Let $X$ and $Y$ be $n\times n$ matrices with rank $x$ and $y$ respectively. Let the dimension of $\mathrm{col}(X) \cap \mathrm{col}(Y)$ be $c$ and the dimension of $\mathrm{row}(X) \cap \mathrm{row}(Y)$ be $r$. Then the rank of $X-Y$ is at least $x+y-c-r$.
\end{lemma}

We can now use Lemma \ref{lem:col_row_rank} to obtain an upper bound on $P_q^n(i,j,k)$ in terms of $M_q^n(i)$ and $Q_q^n(i,j,c)$:
\begin{corollary}\label{cor:P_upper}
    If $i \geq j$ and $i+j \geq k$ then
    \begin{equation*}
        P_q^n(i,j,k) \leq M_q^n(i) \cdot 2\sum_{c=\lceil\frac{i+j-k}{2}\rceil}^jQ_q^n(i,j,c).
    \end{equation*}
\end{corollary}
\begin{proof}
    Pick an arbitrary matrix $X \in \mathbb{F}_q^{n \times n}$ of rank $i$, there are $M_q^n(i)$ such choices. Now, we will count the amount of matrices $Y \in \mathbb{F}_q^{n \times n}$ such that $Y$ is of rank $j$ and $\rk(X-Y) \leq k$. By Lemma \ref{lem:col_row_rank}, we know that this means either the dimension of $\mathrm{col}(X) \cap \mathrm{col}(Y)$ or the dimension of $\mathrm{row}(X) \cap \mathrm{row}(Y)$ is at least $\frac{i+j-k}{2}$. Let us assume we are in the former case, the latter works analogously.
    
    Denote by $c$ the dimension of $\mathrm{col}(X) \cap \mathrm{col}(Y)$. Then, the total amount of matrices $Y \in \mathbb{F}_q^{n \times n}$ which satisfy $\rk(Y) = j$ and $\rk(X-Y) \leq k$ is
    \begin{equation*}
        \sum_{c=\lceil\frac{i+j-k}{2}\rceil}^j Q_q^n(i,j,c) .
    \end{equation*}
    Using an analogous argument as above, we find this same amount of choices when we look at the case where $\mathrm{dim}(\mathrm{row}(X) \cap \mathrm{row}(Y)) \geq \frac{i+j-k}{2}$. Thus
    \begin{equation*}
        P_q^n(i,j,k) \leq M_q^n(i) \cdot 2 \sum_{c=\lceil\frac{i+j-k}{2}\rceil}^j Q_q^n(i,j,c), 
    \end{equation*}
    and the result follows.
\end{proof}

 We now derive an expression for $Q_q^n(i,j,c)$ (step 4) using the following lemma.
\begin{lemma}\cite[Lemma 2.1]{wang_association_2010}\label{lem:intersect_spaces}
    Let $U \subset \mathbb{F}_q^n$ be a $i$-dimensional subspace. Then the amount of $j$-dimensional subspaces $V \subset\mathbb{F}_q^n$ for which $\dim(U \cap V) =c$ is equal to
    \begin{equation*}
        q^{(i-c)(j-c)}\binom{i}{c}_q\binom{n-i}{j-c}_q.
    \end{equation*}
\end{lemma}

\begin{corollary}\label{cor:Q_expression}
    If $c \leq j$ 
    \begin{equation*}
        Q_q^n(i,j,c) = q^{(i-c)(j-c)}\binom{i}{c}_q\binom{n-i}{j-c}_q \prod_{\ell=0}^{j-1}(q^n-q^\ell).
    \end{equation*}
\end{corollary}
\begin{proof}
    Fix an $X \in \mathbb{F}_q^{n \times n}$ of rank $i$. Given a value of $c$ we want to count all matrices $Y \in \mathbb{F}_q^{n \times n}$ for which $\rk(Y) = j$ and the dimension of $\mathrm{col}(X) \cap \mathrm{col}(Y)$ is $c$. To do this, we first count all possible column spaces $Y$ could have. By Lemma \ref{lem:intersect_spaces}, we find that for a fixed $c$ there are
    \begin{equation*}
        q^{(i-c)(j-c)}\binom{i}{c}_q\binom{n-i}{j-c}_q
    \end{equation*}
    different choices of $\mathrm{col}(Y)$. Each column space $U$ has $\prod_{\ell=0}^{j-1}(q^n-q^\ell)$ different matrices $Y$ which have $U$ as column space. Hence,
    \begin{equation*}
        Q_q^n(i,j,c) = q^{(i-c)(j-c)}\binom{i}{c}_q\binom{n-i}{j-c}_q \prod_{\ell=0}^{j-1}(q^n-q^\ell). \qedhere
    \end{equation*}
\end{proof}

We now split into two cases depending on whether $i+j \geq k$ or $i+j < k$. In the former case, we need the following lemma.

\begin{lemma}\label{lem:Q_eps}
    Let $k=\alpha n$ for some fixed $0<\alpha < \frac{2}{3}$. Let $j \leq i \leq k$ and $i+j > k$. Then there exists a $0<\delta<1$ depending only on $q$ and $\alpha$, such that, for large enough $n$, 
    \begin{equation*}
        \sum_{c=\lceil\frac{i+j-k}{2}\rceil}^jQ_q^n(i,j,c) \leq (M_q^n(j))^\delta.
    \end{equation*}
\end{lemma}
\begin{proof}
    Corollary \ref{cor:Q_expression} tells us
    \begin{equation*}
        Q_q^n(i,j,c) = q^{(i-c)(j-c)}\binom{i}{c}_q\binom{n-i}{j-c}_q \prod_{\ell=0}^{j-1}(q^n-q^\ell).
    \end{equation*}
    In particular, we find, when $c>i+j-n$,
    \begin{equation*}
        Q_q^n(i,j,c) = \Theta(q^{(i-c)(j-c)+c(i-c)+(j-c)(n-i-(j-c))+jn}) = \Theta(q^{c(i+j-n-c)+j(2n-j)}).
    \end{equation*}
    Note that for smaller $c$ we have $Q_q^n(i,j,c) = 0$.
    
    We can see that $c(i+j-n-c)$ is negative and decreasing when $c > i+j-n$. We further observe that $\frac{i+j-k}{2} < i+j-n$ by $\alpha<\frac{2}{3}$, and hence $c = \lceil\frac{i+j-k}{2}\rceil$ will be the maximal case. So we will be interested in when $c = \frac{i+j-k}{2}>0$. Finally, we can see $\frac{i+j-k}{2}(i+j-n-\frac{i+j-k}{2})+j(2n-j)$ is maximal when $i=j=k=\alpha n$. After substitution this gives us
    \begin{equation*}
        c(i+j-n-c)+j(2n-j) \leq \left(\frac{3\alpha}{2}-\frac{\alpha^2}{4}\right)n^2.
    \end{equation*}
    On the other hand, $M_q^n(k) = \Theta(q^{(2\alpha-\alpha^2)n^2})$.
    Now by choosing $\delta > \frac{6-\alpha}{8-4\alpha}$ (note $\alpha <\frac{2}{3} \implies \frac{6-\alpha}{8-4\alpha}<1$) we find
    \begin{equation*}
        Q_q^n(i,j,c) < (M_q^n(j))^\delta
    \end{equation*}
    for all $c \in [\lceil\frac{i+j-k}{2}\rceil,j]$, and hence
    \begin{equation*}
        \sum_{c=\lceil\frac{i+j-k}{2}\rceil}^jQ_q^n(i,j,c) \leq (M_q^n(j))^\delta. \qedhere
    \end{equation*}    
\end{proof}

We now have all the preliminary results that we need to prove Theorem \ref{thm:matrix_size_gv}.

\begin{proof}[Proof of Theorem \ref{thm:matrix_size_gv}]
    The main thing we must show to prove Theorem \ref{thm:matrix_size_gv} is that $T \leq D^{2-\varepsilon}$ for some $\varepsilon>0$. We will do this by using the bound from Lemma \ref{lem:TR_Rprime_ineq}, and then bounding each $P_q^n(i,j,k)$ in terms of $M_q^n(k)$.

    Fix a $\beta>0$ such that $\frac{1}{2}\alpha <\beta<\alpha$. We must split into two cases, $i+j \geq 2\beta n$ and $i+j < 2\beta n$:
    \begin{description}
        \item[Case $i+j \geq 2\beta n$.] In this case, $i+j \geq 2\beta n >\alpha n =k$, hence Lemma \ref{lem:Q_eps} tells us 
        \begin{equation*}
            \sum_{c=\lceil\frac{i+j-k}{2}\rceil}^jQ_q^n(i,j,c) \leq (M_q^n(j))^\delta.
        \end{equation*}
        In combination with Corollary \ref{cor:P_upper} we obtain
    \begin{equation}\label{eq:Rprime_R_bound_1}
        P_q^n(i,j,k) \leq M_q^n(j)^{2-\varepsilon_1} \leq M_q^n(k)^{2-\varepsilon_1},
    \end{equation}
    for some $\varepsilon_1>0$

    \item[Case $i+j < 2\beta n$.] We consider the trivial bound $P_q^n(i,j,k) \leq M_q^n(i) \cdot M_q^n(j)$, and find that asymptotically
    \begin{align*}
        M_q^n(i) \cdot M_q^n(j) &= \Theta(q^{2n(i+j)-i^2-j^2}), \\
        M_q^n(k) &= \Theta(q^{2nk-k^2}).
    \end{align*}
    We note that $2n(i+j)-i^2-j^2$ is maximal when $i=j=\beta n$, and hence if we choose
    \begin{equation*}
        \varepsilon_2 < \frac{4(\alpha-\beta)-2(\alpha^2-\beta^2)}{2\alpha-\alpha^2},
    \end{equation*}
    we will have
    \begin{equation}\label{eq:Rprime_R_bound_2}
       P_q^n(i,j,k) \leq M_q^n(i) \cdot M_q^n(j) < M_q^n(k)^{2-\varepsilon_2}
    \end{equation}
    for large enough $n$.
    \end{description}

    Now set $\varepsilon= \min\{\varepsilon_1,\varepsilon_2\}$, then, by Lemma \ref{lem:TR_Rprime_ineq} and the bounds on $P_q^n(i,j,k)$ from equations \eqref{eq:Rprime_R_bound_1} and \eqref{eq:Rprime_R_bound_2}:
    \begin{align*}
         T(\Gamma(\spacc{n_n}{m_n}{q})^k) &\leq 2q^{2\sum_{\ell=2}^t n_\ell m_\ell} \sum_{i=0}^k\sum_{j=k-i}^i P_q^n(i,j,k) \\
         &\leq 2q^{2\sum_{\ell=2}^t n_\ell m_\ell} k^2M_q^n(k)^{2-\varepsilon}\\
         &\leq M_q^n(k)^{2-\varepsilon} \\
         &\leq D(\Gamma(\spacc{n_n}{m_n}{q})^k)^{2-\varepsilon},
    \end{align*}
    for a sufficiently small $\varepsilon>0$ and large enough $n$, where we used that $\sum_{\ell=2}^t n_\ell m_\ell = O(\log_q n)$. Thus, since we just showed that $T \leq D^{2-\varepsilon}$, we can apply Theorem \ref{thm:alpha_logbound} and Proposition \ref{prop:Aq_alpha} to obtain the desired result
    \begin{equation*}
        A_q^{SRK}(\mathbf{n_n},\mathbf{m_n},k+1) \geq \varepsilon\frac{\abs{V}}{20D}\log_2 D.
    \end{equation*}
    The stronger statement about partitioning $\spacc{n_t}{m_t}{q}$ into sum-rank-metric codes of average size larger than the above bound follows directly from Theorem \ref{thm:chi_logbound}, taking $f=D^\varepsilon$.    
\end{proof}

We note that the approach presented in this section can likely be used to extend Theorem \ref{thm:matrix_size_gv} to sum-rank metrics with one $n \times m$ matrix where $n$ is constant or grows slower than $m$. In this case we expect a similar result to Theorem \ref{thm:matrix_size_gv} to hold as well. An extension to sum-rank metrics with two or more $n \times n$ matrices is not as straightforward, and will require adjusting the approach used to bound $T$. 

\section{A connection with set-coloring Ramsey numbers}\label{sec:SRK_Ramsey}

In \cite{CFHMSV22} a connection was established between Hamming-metric codes and set-coloring Ramsey numbers, by showing that the former provides a lower bound for the latter. This was expanded upon in \cite{CFPZ2024} by considering what happens near the zero-rate threshold. The authors of the second paper were able to show that the aforementioned bound is almost tight under this condition. Using the fact that sum-rank-metric codes and Hamming-metric codes can be related to one another as we saw in Section \ref{sec:ham_srk_connection}, we are able to extend the results in \cite{CFPZ2024} for sum-rank-metric codes. These will follow quite directly from Theorem \ref{thm:srk_ham_leq} and Corollary \ref{cor:srk_ham_geq}. In particular, we extend \cite[Theorems 1.2 and 1.3]{CFPZ2024}.

Recall that we denote by $R(k;r,s)$ a \emph{set-coloring Ramsey number}, which is the minimum positive integer $n$ such that in every edge-coloring $\chi:E(K_n) \rightarrow \binom{r}{s}$, where each edge is mapped onto a set of $s$ colors chosen from a palette of $r$ colors, there exists a monochromatic $k$-clique. In this section, we will connect sum-rank metric codes to set-coloring Ramsey numbers using results from \cite{CFHMSV22} and \cite{CFPZ2024}.

Set-coloring Ramsey numbers were first connected to Hamming-metric codes in \cite{CFHMSV22} as follows.
\begin{theorem}\cite[Theorem 2.1]{CFHMSV22}\label{thm:ham_ram}
    Let $a,b,d,N$ be positive integers with $b < a$ and $d < N$ and let $r =Na$ and $s = db$. Then, for $q = R(k;a,b)-1$,
    \begin{equation*}
        A_q(N,d) < R(k;r,s).
    \end{equation*}
\end{theorem}

The tightness of this bound was further investigated in \cite{CFPZ2024}. Of particular interest to them, was the case when $d$ is close to the \emph{zero-rate threshold}. That is, when $d$ is close to $(1-\frac{1}{q})N$.

\begin{theorem}\cite[Theorem 1.2]{CFPZ2024}
    For any positive integer $q$, and $\varepsilon > 0$, there is a $c>0$ such that if $N,d$ are positive integers with $d \leq (1-\frac{1}{q})t$ and $j = (1-\frac{1}{q})t - d + 1$, then
    \begin{equation*}
        R(q+1;t,d) \leq \max((1+\varepsilon)A_{q}(N,d-cj),\varepsilon d).
    \end{equation*}
\end{theorem}

Using the connection to the size of the largest Hamming-metric code from Theorem \ref{thm:srk_ham_leq}, the Ramsey-type bound from Theorem \ref{thm:ham_ram} was extended to the sum-rank-metric case in \cite{AAA2024}. 

\begin{theorem} \cite{AAA2024}\label{thm:SRK_Ram_LB}
Let $m=\max\{m_1,\dots,m_t\}$. Let $a,b,d,N$ be positive integers with $b < a$ and $d < N$ and let $r =Na$ and $s = db$. If $q^m=R(k;a,b)-1$
then there exist constants $c,c'>0$ such that, for any integers $k,r,s$ with $k\geq3$ and $r>s\geq1$ we have 
\begin{equation}\label{eq:srk_ramsey}
    A_q^{SRK}(\mathbf{n},\mathbf{m},d) < R(k;r,s)\leq 2^{c'k(r-s)^2r^{-1}\log\left(\frac{r}{\min(s,r-s)}\right)}.
\end{equation}
\end{theorem}

 We will now be extending several results relating Hamming-metric codes to set-coloring Ramsey numbers, in particular those in \cite{CFPZ2024}, to sum-rank-metric codes.

\begin{corollary}
    Let $k$ be a positive integer, and let $m=\max\{m_1,\dots,m_t\}$. Let $j \leq \sqrt{N\frac{k-1}{q^m-1}}$ and let $\mathcal{C}\subseteq\spac$ be a sum-rank-metric code with $|\mathcal{C}|\geq 2$ and $(1-1/q^m)(N-j)<N$. Then,
    \begin{equation*}
        A_q^{SRK}(\mathbf{n},\mathbf{m},(1-\frac{1}{q^m})(N-j)) = O_{q^m,k}(N^k).
    \end{equation*}
\end{corollary}
\begin{proof}
    Apply Theorem \ref{thm:srk_ham_leq} followed by \cite[Theorem 1.3]{CFPZ2024}.
\end{proof}

In view of Corollary \ref{cor:srk_ham_geq}, we can also make use of the upper bounding results.
\begin{corollary}\label{cor:SRK_Ram_UB}
    Let $\mathbf{n} = (n_1,\hdots,n_t)$ and $\mathbf{m} = (m_1,\hdots,m_t)$ with $m' = \min\{m_i\}$. For any positive integer $q$, and $\varepsilon > 0$, there is a $c>0$ such that if $d$ is a positive integer with $d \leq (1-\frac{1}{q^{m'}})t$ and $j = (1-\frac{1}{q^{m'}})t - d + 1$, then
    \begin{equation*}
        R(q^{m'}+1;t,d) \leq \max((1+\varepsilon)A_{q}^{SRK}(\mathbf{n},\mathbf{m},d-cj),\varepsilon d).
    \end{equation*}
\end{corollary}
\begin{proof}
    Apply Corollary \ref{cor:srk_ham_geq} followed by \cite[Theorem 1.2]{CFPZ2024}.
\end{proof}

Unlike for the Hamming case, the set-coloring Ramsey number considered in Theorem \ref{thm:SRK_Ram_LB} and Corollary \ref{cor:SRK_Ram_UB} are considerably different for most $\mathbf{n}$, $\mathbf{m}$. So we cannot say anything about near-tightness as it was done for the Hamming case in \cite{CFPZ2024} (since our connection between the sum-rank metric and the set-coloring Ramsey number goes through the Hamming metric). Because the sum-rank metric and Hamming metric are in general not close to being equivalent (in particular, while $f$ as seen in \eqref{eq:srk_to_Hamming} is injective, it is not bijective in general), most results from \cite{CFPZ2024} do not extend to the sum-rank metric.

\section{Concluding remarks}\label{sec:concludingremarks}

In this paper we proved an asymptotic improvement for the GV bound for some cases of the sum-rank metric, and even showed a stronger result that states that we can partition the sum-rank-metric ambient space into codes whose average size matches these bounds improvements, see Theorem \ref{thm:srk_asymp_gv} and Theorem \ref{thm:matrix_size_gv}. Moreover, we extended several results from \cite{CFPZ2024} on the connection between Hamming-metric codes and set-coloring Ramsey numbers to the sum-rank metric.

A future line of research may be that of extending Theorems \ref{thm:srk_asymp_gv} and \ref{thm:matrix_size_gv} to more types of sum-rank metrics. The approach presented in Theorem \ref{thm:srk_asymp_gv} allows us to improve the GV bound asymptotically for a large number of instances of the sum-rank metric, with the only requirement being that $\mathbf{n}$ ends in a large enough amount of $1$'s and $\mathbf{m}$ ends in a large amount of a single constant. This condition cannot be dropped with our approach, as we do not in general have $D \sim D_H$ and $T \sim T_H$. In particular, to extend our result from Theorem \ref{thm:srk_asymp_gv} one would need to use new estimations on $D$ and $T$. A different approach that could possibly be used to extend Theorem \ref{thm:srk_asymp_gv} is the framework provided in \cite{kim_exponential_2022}. In this paper, the authors made use of a probabilistic point of view to present a more general approach to proving asymptotic improvements for the GV bound, which they then applied to several coding metrics. However, such conditions are quite hard to work with for the sum-rank case, in part because we still need to work with $D$, the degree of a vertex in the power of the sum-rank-metric graph $\Gamma(\spac)^k$. Moreover, this $D$ becomes quite unruly in the general case, as can be seen in the expression derived in \cite[Lemma III.5]{byrne_fundamental_2021}. Similarly, the result from Theorem \ref{thm:matrix_size_gv} could also be extended to to sum-rank metrics with non-square matrix growing at rate $n$, or sum-rank metrics with more than one component growing at rate $n$,  as mentioned at the end of Section \ref{sec:asymptoticsizematrix}. However, the latter becomes increasingly complex if one uses an approach similar to ours.

\subsection*{Acknowledgments}

Aida Abiad is supported by NWO (Dutch Research Council) through the grants \linebreak VI.Vidi.213.085 and OCENW.KLEIN.475. Harper Reijnders is supported through the grant VI.Vidi.213.085. Michael Tait is partially supported by the US National Science Foundation via the grant DMS-2245556.

\bibliographystyle{abbrv}

\end{document}